\documentclass[a4paper]{amsart}

\usepackage[english]{babel}
\usepackage{amsthm, latexsym, gastex, amssymb,url,listings}
\usepackage[mathcal]{eucal}
\usepackage{graphicx,url}

\copyrightinfo{2014}{Simone Ugolini}

\DeclareMathOperator{\Tr}{Tr}

\DeclareMathOperator{\Gr}{Gr}

\renewcommand{\phi}[0]{\varphi}
\renewcommand{\theta}[0]{\vartheta}
\renewcommand{\epsilon}[0]{\varepsilon}

\newcommand{\Pro}{\text{$\mathbf{P}^1$}}
\newcommand{\F}{\text{$\mathbf{F}$}}

\newtheorem{theorem}{Theorem}[section]
\newtheorem{lemma}[theorem]{Lemma}

\theoremstyle{definition}

\newtheorem{definition}[theorem]{Definition}
\newtheorem{example}[theorem]{Example}

\theoremstyle{remark}
\newtheorem{remark}[theorem]{Remark}

\numberwithin{equation}{section}

\begin{document}

\bibliographystyle{amsplain}

\date{}

\keywords{Finite fields, polynomials, sequences}
\subjclass[2010]{11R09, 11T55, 12E05}
\title[]
{Sequences of irreducible polynomials without prescribed coefficients over finite fields of even characteristic}

\author{S.~Ugolini}
\email{sugolini@gmail.com} 

\begin{abstract}
In this paper we deal with the construction of sequences of irreducible polynomials with coefficients in finite fields of even characteristic. We rely upon a transformation used by Kyuregyan in 2002, which generalizes the $Q$-transform employed previously by Varshamov and Meyn for the synthesis of irreducible polynomials. While in the iterative procedure described by Kyuregyan the coefficients of the initial polynomial of the sequence have to satisfy certain hypotheses, in the present paper we construct infinite sequences of irreducible polynomials of non-decreasing degree starting from any irreducible polynomial.
\end{abstract}

\maketitle
\section{Introduction}
In the last decades many investigators dealt with the iterative construction of sequences of irreducible polynomials of non-decreasing degree with coefficients over finite fields. One of the possible construction relies on the so-called $Q$-transform, which takes any polynomial $f$ of positive
degree $n$ to $f^Q(x) = x^n \cdot f(x+x^{-1})$. Among others, the construction of irreducible polynomials via the $Q$-transform was studied by  Varshamov and Garakov \cite{var} and later by Meyn \cite{mey}. 

Adopting the notations of \cite{mey}, the reciprocal $f^*$ of a polynomial $f$ of degree $n$ is the polynomial $f^*(x) = x^n \cdot f(1/x)$. If $f=f^*$, then we say that $f$ is self-reciprocal. In general, the $Q$-transform $f^Q$ of any polynomial $f$ is self-reciprocal.

If $\alpha$ is an element of the field $\F_{2^n}$ with $2^n$ elements for some positive integer $n$, then the absolute trace of $\alpha$ is 
\begin{equation*}
\Tr_n (\alpha) = \sum_{i=0}^{n-1} \alpha^{2^i}.
\end{equation*}
We remind the reader that $\Tr_n(\alpha) \in \{0, 1\}$.

The following result plays a crucial role for the synthesis of sequences of irreducible polynomials over finite fields of even characteristic.

\begin{theorem}\cite[Theorem 9]{mey}\label{mey_bas_2}
The $Q$-transform of a self-reciprocal irreducible monic polynomial $f(x) = x^n + a_1 x^{n-1} + \dots + a_1 x +1 \in \F_{2^k} [x]$ with $\Tr_n (a_1) = 1$ is a self-reciprocal irreducible monic polynomial of the same kind, i.e. $f^Q (x) = x^{2n} + \tilde{a}_1 x^{2n-1} + \dots + \tilde{a}_1 x+ 1$ satisfies $\Tr_n (\tilde{a}_1) = 1$. 
\end{theorem} 

Relying upon Theorem \ref{mey_bas_2} Meyn shows, after \cite[Example 3, page 50]{mey}, how to construct iteratively a sequence  $\{f_i \}_{i \geq 0}$ of irreducible polynomials in $\F_{2}[x]$, starting from any monic irreducible polynomial $f_0 =x^n + a_{n-1} x^{n-1} + \dots + a_1 x + a_0 \in \F_2 [x]$ such that $a_{n-1} = a_1 = 1$. The polynomials of the sequence are inductively defined as $f_i := f_{i-1}^Q$ for any positive integer $i$. Moreover, according to Theorem \ref{mey_bas_2}, the degree of $f_i$ is twice the degree of $f_{i-1}$ for any positive integer $i$.

In \cite{seq2} we concentrated on the construction of sequences of irreducible polynomials with coefficients in $\F_2$ removing any assumption on the coefficients of the initial polynomial of the sequence. The drawback of relaxing the hypotheses on the initial polynomial is that we could face a finite number of factorizations of a polynomial in two equal-degree polynomials in $\F_2 [x]$. The number of factorizations depends on the greatest power of $2$ which divides the degree of the initial polynomial of the sequence, as explained in \cite[Section 3.2]{seq2}.

In \cite{kyu_even} Kyuregyan introduced a more general construction of sequences of irreducible polynomials having coefficients in finite fields of even characteristic. Such a construction is based on the transformations which take a polynomial $f$ of degree $n$ to the polynomial $x^n \cdot f(x+ \delta^2 x^{-1})$, for some non-zero element $\delta$ in the field of the coefficients of $f$. For the sake of clarity we introduce a notation for this family of transformations.

\begin{definition}
If $f$ is a polynomial of positive degree $n$ in $\F_{2^s} [x]$, for some positive integer $s$, and $\alpha \in \F_{2^s}^*$, then the $(Q, \alpha)$-transform of $f$ is 
\begin{equation*}
f^{(Q, \alpha)} (x) = x^n \cdot f(x + \alpha x^{-1}). 
\end{equation*} 
\end{definition}  
\begin{remark}
For $\alpha = 1$ the $(Q, \alpha)$-transform coincides with the $Q$-transform. 
\end{remark}  

The forthcoming theorem proved by Kyuregyan \cite{kyu_even} furnishes an iterative procedure for constructing sequences of irreducible polynomials with coefficients in finite fields of even characteristic. The procedure is based on the $(Q, \alpha)$-transforms and requires that some hypotheses on the initial polynomial of the sequence are satisfied. We state \cite[Theorem 3]{kyu_even} using the notation introduced in the present paper. Actually, we adapt the statement of the theorem as presented in \cite{kyu_ind}.

\begin{theorem}\cite[Proposition 3]{kyu_ind}
Let $\delta \in \F_{2^s}^*$ and $F_1 (x) = \sum_{u=0}^n c_u x^u$ be an irreducible polynomial over $\F_{2^s}$ whose coefficients satisfy the conditions 
\begin{equation*}
\Tr_s \left( \frac{c_1 \delta}{c_0} \right) = 1 \quad \text{ and } \quad \Tr_s \left( \frac{c_{n-1}}{\delta} \right) = 1.
\end{equation*}
Then all members of the sequence $(F_k(x))_{k \geq 1}$ defined by 
\begin{equation*}
F_{k+1} (x) = F_k^{(Q, \delta^2)} (x), \quad k \geq 1
\end{equation*}
are irreducible polynomials over $\F_{2^s}$.
\end{theorem}

In the present paper we aim to construct sequences of irreducible polynomials $\{ f_k \}_{k \geq 0}$, where the initial polynomial $f_0$ of the sequence is monic and irreducible, but no  extra assumption is made on its coefficients.

While our previous investigation \cite{seq2} was based on the dynamics of the map $x \mapsto x+x^{-1}$ over finite fields of characteristic $2$, studied by us in \cite{SU2}, the starting point for the present paper is the dynamics of the maps
\begin{displaymath}
\begin{array}{rcl}
\theta_{\alpha} : 
x & \mapsto & 
\begin{cases}
\infty & \text{if $x \in \{ 0, \infty \}$,}\\
x + \alpha x^{-1} & \text{otherwise}
\end{cases}
\end{array}
\end{displaymath}
over $\Pro(\F_{2^s}) = \F_{{2^s}} \cup \{ \infty \}$, for any positive integer $s$ and any choice of $\alpha \in \F_{2^s}^*$.

Actually, once we know the dynamics of the map $\theta_1: x \mapsto x + x^{-1}$,  we can transfer the results to all the maps $\theta_{\alpha}$. Consider in fact the map defined over $\Pro(\F_{2^s}) = \F_{{2^s}} \cup \{ \infty \}$, for any positive integer $s$ and for any $\gamma \in \F_{2^s}^*$, as follows:
\begin{displaymath}
\begin{array}{rcl}
\psi_{\gamma} : 
x & \mapsto & 
\begin{cases}
\infty & \text{if $x = \infty$,}\\
\gamma \cdot x & \text{otherwise.}
\end{cases}
\end{array}
\end{displaymath}

If $\gamma = \sqrt{\alpha}$, namely $\gamma$ is the square root of $\alpha$ for a generic element $\alpha \in \F_{2^s}^*$, then 
\begin{equation*}
\theta_{\alpha}  = \psi_{\gamma} \circ \theta_{1} \circ \psi_{\gamma^{-1}}.
\end{equation*}

We can construct a graph $\Gr_s (\alpha)$ related to the dynamics of the map $\theta_{\alpha}$ over $\Pro(\F_{2^s})$ as in \cite{SU2}. The vertices of the graph are labelled by the elements of $\Pro (\F_{2^s})$ and an arrow joins a vertex $\beta$ to a vertex $\gamma$ if $\gamma = \theta_{\alpha} (\beta)$. The graph $\Gr_s (\alpha)$ is isomorphic to $\Gr_s (1)$. As in \cite{SU2} we say that an element $\beta \in \Pro(\F_{2^s})$ is $\theta_{\alpha}$-periodic if $\theta_{\alpha}^i (\beta) = \beta$ for some positive integer $i$. Moreover, since $\Pro (\F_{2^s})$ is finite, any non-$\theta_{\alpha}$-periodic element belonging to  $\Pro (\F_{2^s})$ is preperiodic.  

In Section \ref{graphs_section} we briefly describe some properties of the graphs $\Gr_s (\alpha)$, while in Section \ref{sequences_section} we study the sequences of irreducible polynomials generated by the iterations of the $(Q, \alpha)$-transforms.

\section{The structure of the graphs $\Gr_s (\alpha)$}\label{graphs_section}
Since all the graphs  $\Gr_s (\alpha)$ are isomorphic to  $\Gr_s (1)$, from \cite{SU2} we deduce the following, for a chosen element $\alpha \in \F_{2^s}^*$:
\begin{itemize}
\item every connected component of $\Gr_s (\alpha)$ is formed by a cycle whose vertices are roots of binary trees of the same depth;
\item either all the trees of a connected component of $\Gr_s (\alpha)$ have depth $1$ or all the trees of that component have depth $l+2$, being $2^l$ the greatest power of $2$ which divides $s$ \cite[Lemma 4.3, Lemma 4.4]{SU2}.
\end{itemize}

\begin{example}\label{graph_6}
In this example we construct the graph $\Gr_6(\alpha)$, where $\alpha$ is a root of the Conway polynomial $x^6+x^4+x^3+x+1$, which is primitive in $\F_2 [x]$. The labels of the vertices are $\infty$, `0' (the zero of $\F_2$) and the exponents $i$ of the powers $\alpha^i$, for $0 \leq i \leq 62$.

We notice that the graph $\Gr_6(\alpha)$ is isomorphic to the graph $\Gr_6(1)$, which the reader can find in \cite[Example 4.2]{seq2}. Since the greatest power of $2$ dividing $6$ is $2$, the trees belonging to a connected component of $\Gr_{6} (\alpha)$ either have depth $3$ or $1$.

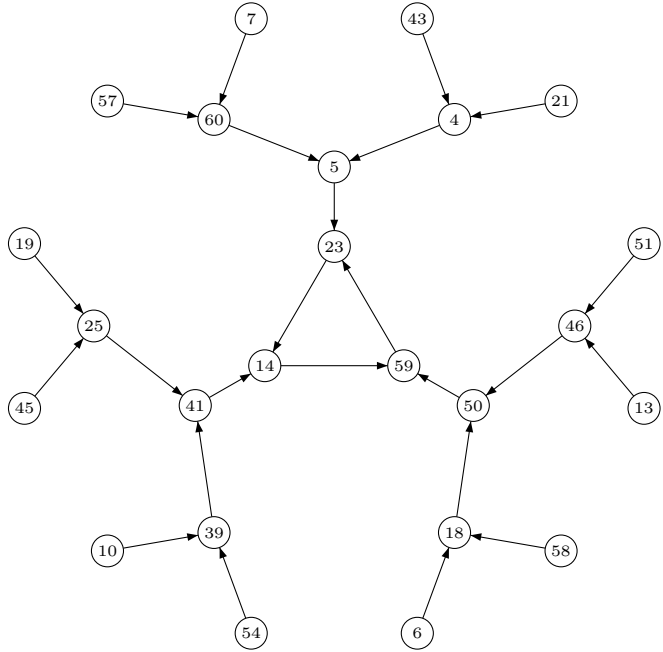
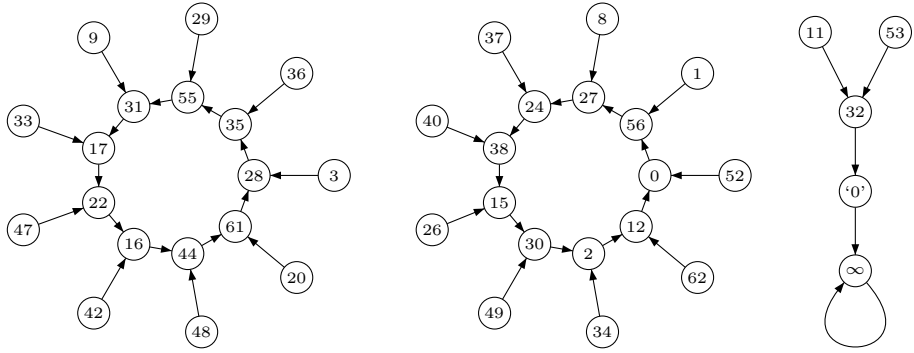
\begin{center}
\unitlength=3pt
    \begin{picture}(86, 86)(-43,-43)
    \gasset{Nw=4,Nh=4,Nmr=2,curvedepth=0}
    \thinlines
    \tiny
    \node(N1)(0,10){$23$}
    \node(N2)(-8.67,-5){$14$}
    \node(N3)(8.67,-5){$59$}

    \node(N11)(0,20){$5$}
    \node(N21)(-17.34,-10){$41$}
    \node(N31)(17.34,-10){$50$}
    
    \node(N112)(-15,25.98){$60$}
    \node(N211)(-30,0){$25$}
    \node(N212)(-15,-25.98){$39$}
    \node(N311)(15,-25.98){$18$}
    \node(N312)(30,0){$46$}
    \node(N111)(15,25.98){$4$}
    
    \node(N1121)(-10.36,38.64){$7$}
    \node(N1122)(-28.29,28.29){$57$}
    \node(N2111)(-38.64,10.36){$19$}
    \node(N2112)(-38.64,-10.36){$45$}
    \node(N2121)(-28.29,-28.29){$10$}
    \node(N2122)(-10.36,-38.64){$54$}
    \node(N3111)(10.36,-38.64){$6$}
    \node(N3112)(28.29,-28.29){$58$}
    \node(N3121)(38.64,-10.36){$13$}
    \node(N3122)(38.64,10.36){$51$}
    \node(N1111)(28.29,28.29){$21$}
    \node(N1112)(10.36,38.64){$43$}
   
    \drawedge(N1,N2){}
    \drawedge(N2,N3){}
    \drawedge(N3,N1){}
    
    \drawedge(N11,N1){}
    \drawedge(N21,N2){}
    \drawedge(N31,N3){}
    
    \drawedge(N111,N11){}
    \drawedge(N112,N11){}
    \drawedge(N211,N21){}
    \drawedge(N212,N21){}
    \drawedge(N311,N31){}
    \drawedge(N312,N31){}
    
    \drawedge(N1111,N111){}
    \drawedge(N1112,N111){}
    \drawedge(N1121,N112){}
    \drawedge(N1122,N112){}
    \drawedge(N2111,N211){}
    \drawedge(N2112,N211){}
    \drawedge(N2121,N212){}
    \drawedge(N2122,N212){}
    \drawedge(N3111,N311){}
    \drawedge(N3112,N311){}
    \drawedge(N3121,N312){}
    \drawedge(N3122,N312){}
\end{picture}
\end{center}
\begin{center}
    \unitlength=3pt
    \begin{picture}(80, 48)(-5,-24)
    \gasset{Nw=4,Nh=4,Nmr=2,curvedepth=0}
    \thinlines
    \tiny
    \node(N1)(10,0){$28$}
    \node(N2)(7.66,6.42){$35$}
    \node(N3)(1.73,9.84){$55$}
    \node(N4)(-5,8.66){$31$}
    \node(N5)(-9.39,3.42){$17$}
    \node(N6)(-9.39,-3.42){$22$}
    \node(N7)(-5,-8.66){$16$}
    \node(N8)(1.73,-9.84){$44$}
    \node(N9)(7.66,-6.42){$61$}
    
    \node(N11)(20,0){$3$}
    \node(N12)(15.32,12.84){$36$}	
    \node(N13)(3.46,19.68){$29$}
    \node(N14)(-10,17.32){$9$}
    \node(N15)(-18.78,6.84){$33$}
    \node(N16)(-18.78,-6.84){$47$}
    \node(N17)(-10,-17.32){$42$}
    \node(N18)(3.46,-19.68){$48$}
    \node(N19)(15.32,-12.84){$20$}

    \drawedge(N1,N2){}
    \drawedge(N2,N3){}
    \drawedge(N3,N4){}
    \drawedge(N4,N5){}
    \drawedge(N5,N6){}
    \drawedge(N6,N7){}
    \drawedge(N7,N8){}
    \drawedge(N8,N9){}
    \drawedge(N9,N1){}
    
    \gasset{curvedepth=0}
     
    \drawedge(N11,N1){}
    \drawedge(N12,N2){}
    \drawedge(N13,N3){}
    \drawedge(N14,N4){}
    \drawedge(N15,N5){}
    \drawedge(N16,N6){}
    \drawedge(N17,N7){}
    \drawedge(N18,N8){}
    \drawedge(N19,N9){}

    \node(L1)(60,0){$0$}
    \node(L2)(57.66,6.42){$56$}
    \node(L3)(51.73,9.84){$27$}
    \node(L4)(45,8.66){$24$}
    \node(L5)(40.61,3.42){$38$}
    \node(L6)(40.61,-3.42){$15$}
    \node(L7)(45,-8.66){$30$}
    \node(L8)(51.73,-9.84){$2$}
    \node(L9)(57.66,-6.42){$12$}
    
    \node(L11)(70,0){$52$}
    \node(L12)(65.32,12.84){$1$}
    \node(L13)(53.46,19.68){$8$}
    \node(L14)(40,17.32){$37$}
    \node(L15)(32.22,6.84){$40$}
    \node(L16)(32.22,-6.84){$26$}
    \node(L17)(40,-17.32){$49$}
    \node(L18)(53.46,-19.68){$34$}
    \node(L19)(65.32,-12.84){$62$}

    \drawedge(L1,L2){}
    \drawedge(L2,L3){}
    \drawedge(L3,L4){}
    \drawedge(L4,L5){}
    \drawedge(L5,L6){}
    \drawedge(L6,L7){}
    \drawedge(L7,L8){}
    \drawedge(L8,L9){}
    \drawedge(L9,L1){}
    
    \gasset{curvedepth=0}
     
    \drawedge(L11,L1){}
    \drawedge(L12,L2){}
    \drawedge(L13,L3){}
    \drawedge(L14,L4){}
    \drawedge(L15,L5){}
    \drawedge(L16,L6){}
    \drawedge(L17,L7){}
    \drawedge(L18,L8){}
    \drawedge(L19,L9){}
    
    \node(M1)(85,-12){$\infty$}
    \node(M11)(85,-2){`0'}
    \node(M111)(85,8){$32$}

    \node(M1111)(80,18){$11$}
    \node(M1112)(90,18){$53$}

    \drawedge(M11,M1){}
    \drawedge(M111,M11){}
    \drawedge(M1111,M111){}
    
    \drawedge(M1112,M111){}
	\drawloop[loopangle=-90](M1){}
\end{picture}
\end{center}  
\end{example}

\section{The synthesis of irreducible polynomials via the $(Q, \alpha)$-transforms}\label{sequences_section}
The following lemma about the irreducibility of the polynomials $f^{(Q,\alpha)}$ holds in analogy with \cite[Lemma 4]{mey}.
\begin{lemma}\label{iter_lem}
If $f$ is an irreducible monic polynomial of degree $n$ in $\F_{2^s} [x]$, for some positive integers $n$ and $s$, and $\alpha \in \F_{2^s}^*$, then either $f^{(Q,\alpha)}$ is an irreducible monic polynomial of degree $2n$ in $\F_{2^s} [x]$ or $f^{(Q,\alpha)}$ splits into the product of a pair of irreducible monic polynomials $g_1, g_2$ of degree $n$ in $\F_{2^s} [x]$. In this latter case at least one of $g_1$ and $g_2$ has  no $\theta_{\alpha}$-periodic roots.   
\end{lemma}
\begin{proof}
Let $\beta \in \F_{2^{sn}}$ be a root of $f$ and $\gamma$ a solution of the equation $\theta_{\alpha} (x) = \beta$. Then, $f^{(Q,\alpha)} (\gamma) = \gamma^n \cdot f(\gamma+ \alpha \gamma^{-1}) = 0$. Since $\theta_{\alpha} (\gamma) = \beta$, either $\gamma$ belongs to $\F_{2^{sn}}$ or $\gamma$ belongs to $\F_{2^{2sn}} \backslash \F_{2^{sn}}$. In the latter case $f^{(Q, \alpha)}$ is irreducible of degree $2n$ over $\F_{2^s}$, while in the former case $f^{(Q,\alpha)}$ splits into the product of a pair of irreducible monic polynomials $g_1, g_2$ of degree $n$.

Suppose now that $f^{(Q,\alpha)} (x) = g_1 (x) \cdot g_2 (x)$, where $g_1$ and $g_2$ have degree $n$. We proceed proving that one of $g_1$ and $g_2$ has no $\theta$-periodic roots.  First we notice that $\theta_{\alpha} (x) = \beta$ if and only if $x \in \{\gamma, \alpha \gamma^{-1} \}$. If $\beta$ is not $\theta_{\alpha}$-periodic, then the same holds for $\gamma$ too. Conversely, up to a renaming, $\gamma$ belongs to the first level of the binary tree of $\Gr_{sn} (\alpha)$ rooted in $\beta$. Since $f^{(Q, \alpha)} (\gamma) = 0$, we conclude that either $g_1 (\gamma) = 0$ or $g_2 (\gamma) = 0$. 

Suppose, without loss of generality, that $g_1(\gamma) = 0$. If $\delta$ is any root of $g_1$, then $\gamma = \delta^{2^{si}}$, for some integer $i$. Therefore, $\theta_{\alpha}^k (\delta) = \delta$ for some positive integer $k$ if and only if $(\theta_{\alpha}^k (\delta))^{2^{si}} = \delta^{2^{si}} = \gamma$. Since $(\theta_{\alpha}^k (\delta))^{2^{si}} = \theta_{\alpha}^k (\gamma)$, we conclude that $\delta$ is $\theta_{\alpha}$-periodic if and only if $\gamma$ is $\theta_{\alpha}$-periodic. Hence, all the roots of $g_1$ are not $\theta_{\alpha}$-periodic. 
\end{proof}

The following theorem describes the iterative procedure for the construction of irreducible polynomials over finite fields of even characteristic via the $(Q, \alpha)$-transforms.

\begin{theorem}\label{iter_thm}
Let $f_0 \in \F_{2^s} [x]$, where $s$ is a positive integer, be an irreducible monic polynomial of positive degree $n$. Suppose that $2^{l_s}$ is the greatest power of $2$ which divides $s$, while $2^{l_n}$ is the greatest power of $2$ which divides $n$. Fix an element $\alpha$  in $\F_{2^s}^*$.

Let $f_1$ be one of the at most two irreducible monic polynomials which factor $f_0^{(Q, \alpha)}$. Suppose also that the roots of $f_1$ are not $\theta_{\alpha}$-periodic. 

Consider the sequence of polynomials $\{f_i \}_{i \geq 0}$ constructed inductively setting $f_i$ equal to one of the irreducible monic factors of $f_{i-1}^{(Q,\alpha)}$, for $i \geq 2$.

Then, there exists a positive integer $t \leq l_s+l_n+3$ such that $f_t$ has degree $2n$, while $f_{t+1}$ has degree $4n$. Moreover, for any $i \geq t$, the polynomial $f_i^{(Q,\alpha)}$ is irreducible in $\F_{2^s} [x]$ and the degree of $f_{i+1}$ is twice the degree of $f_i$.
\end{theorem}  
\begin{proof}
Since $f_0$ is an irreducible polynomial of degree $n$ in $\F_{2^s} [x]$, all its roots are in $\F_{2^{sn}}$. Let $\beta_0 \in \F_{2^{sn}}$ be a root of $f_0$. Then, it is possible to construct inductively a sequence of elements $\{\beta_i \}_{i \geq 0}$, where any $\beta_i$ belongs  to an appropriate extension of $\F_{2^{sn}}$, such that:
\begin{itemize}
\item any $\beta_i$ is a root of $f_i$;
\item $\beta_i = \theta_{\alpha} (\beta_{i+1})$.
\end{itemize}
Since $\beta_0 \in \F_{2^{sn}}$ and $\beta_1$ is not $\theta_{\alpha}$-periodic,  all the elements $\beta_i$, for $i \geq 1$, belong to a tree of the graph $\Gr_{2sn} (\alpha)$. In particular, such a tree has depth at least $2$ in $\Gr_{2sn} (\alpha)$. Indeed, there are two possibilities for $\beta_1$: either $\beta_1 \in \F_{2^{sn}}$ or $\beta_1 \in \F_{2^{2sn}} \backslash \F_{2^{sn}}$. In the former case $\beta_1$ lies on a level not smaller than $1$ of a tree in $ \Gr_{sn} (\alpha)$ and consequently $\beta_2$ lies on a level not smaller than $2$ of a tree in $ \Gr_{2sn} (\alpha)$. In the latter case, $\beta_0$ is a leaf of a tree in $\Gr_{sn} (\alpha)$. Such a tree has depth at least $1$. Consequently, $\beta_2$ lies on a level not smaller than $2$ of a tree in $ \Gr_{2sn} (\alpha)$. In both cases we conclude that such a tree has not depth $1$ in $\Gr_{2sn} (\alpha)$, namely it has depth $(1+l_n+l_s)+2$ in $\Gr_{2sn} (\alpha)$ (see Section \ref{graphs_section}). Hence, there exists a positive integer $t \leq l_n + l_s+3$ such that $\beta_t \in \F_{2^{2sn}}$, while $\beta_{t+j} \in \F_{2^{2^{j+1}sn}}$ for any $j \geq 1$. For such an integer $t$ we have that $f_t$ has degree $2n$, while $f_t^{(Q,\alpha)}$ has degree $4n$. More in general, for any $i \geq t$, we have that $f_{i+1} = f_{i}^{(Q, \alpha)}$ and the degree of $f_{i+1}$ is twice the degree of $f_i$.  
\end{proof}
\begin{remark}
In the hypotheses of Theorem \ref{iter_thm} we require that the roots of $f_1$ are not $\theta_{\alpha}$-periodic. Indeed, this is true if $f_1 = f_0^{(Q,\alpha)}$, since in this circumstance the degree of $f_1$ is twice the degree of $f_0$ and consequently the roots of $f_0$ are leaves of $\Gr_{sn}(\alpha)$. 

Consider now the case that $f_0^{(Q,\alpha)} (x) = g_1(x) \cdot g_2 (x)$, for some irreducible monic polynomials $g_1, g_2$ of degree $n$ in $\F_{2^s} [x]$. According to Lemma \ref{iter_lem}, at least one of $g_1$ and $g_2$ has no $\theta_{\alpha}$-periodic roots. Suppose that $g_2$ has no $\theta_{\alpha}$-periodic roots. If we set $f_1 := g_1$ and all the polynomials $f_i$ have degree $n$, for $0 \leq i \leq l_s+l_n+3$, then we break the iterations and set $f_1 := g_2$. Since $g_2$ has no $\theta_{\alpha}$-periodic roots, the hypotheses of Theorem \ref{iter_thm} are satisfied and we can construct inductively a sequence of irreducible monic polynomials, as explained in the theorem.  
\end{remark}

\begin{example}
In this example we construct a sequence of irreducible monic polynomials over $\F_8 [x]$ starting from the polynomial $f_0 (x) = x^4+x+a^3$, being $a$ a root of the primitive polynomial $x^3+x+1 \in \F_2 [x]$. We notice that $f_0$ is irreducible in $\F_8 [x]$ (see \cite[Table 5]{gre}).

We set $\alpha:=a$ and proceed as explained in Theorem \ref{iter_thm}. Adopting the notations of the theorem, in the current example we have that $s=3$, $n=4$, $l_s = 0$ and $l_n =2$. Since $f_0^{(Q,\alpha)}$ is not irreducible, it splits into the product of two irreducible monic polynomials of degree $4$. We set $f_1$ equal to one of the two factors of $f_0^{(Q, \alpha)}$, namely
\begin{equation*}
f_1 (x) := x^4+a^4  x^3 + x^2 + a^2  x + a^6.
\end{equation*}
We notice that $f_1^{(Q, \alpha)}$ is irreducible of degree $8$ and set $f_2 := f_1^{(Q, \alpha)}$. Since $f_2^{(Q,\alpha)}$ is irreducible of degree $16$ in $\F_{8} [x]$, implying that $f_3$ has degree $4n=16$, according to Theorem \ref{iter_thm} all the polynomials $f_i^{(Q, \alpha)}$ are irreducible for $i \geq 2$. Hence, no more factorization is required and we can generate an infinite sequence of irreducible monic polynomials of increasing degree. 
\end{example}

\bibliography{Refs}
\end{document}